\documentclass[final]{dmtcs-episciences}

\usepackage[utf8]{inputenc}
\usepackage{subfigure}

\usepackage{tikz,amsmath,amssymb}
\usepackage[linguistics]{forest}
\tikzstyle{path}=[draw, line width=1.3, color=darkgray]
\tikzstyle{pathlight}=[draw, line width=1, dotted, color=lightgray]

\numberwithin{equation}{section}
\newtheorem{theorem}[equation]{Theorem}
\newtheorem{proposition}[equation]{Proposition}

\newcommand{\G}{\mathcal{G}}
\newcommand{\W}{\mathcal{WOC}}
\newcommand{\V}{\mathcal{V}}
\DeclareMathOperator{\op}{op}
\newcommand{\abs}[1]{\lvert#1\rvert}
\def\col#1{\textcolor{black!25!red}{#1}}

\title{Restricted generating trees for \\weak orderings}
\author{Daniel Birmajer\affiliationmark{1}
  \and Juan B. Gil\affiliationmark{2}
  \and David Kenepp\affiliationmark{3}
  \and Michael D. Weiner\affiliationmark{2}}

\affiliation{
  Nazareth College, Rochester, NY, U.S.A.\\
  Penn State Altoona, Altoona, PA, U.S.A.\\
  Department of Mathematics, UC Davis, Davis, CA, U.S.A.}
\keywords{weak orderings, generating trees, pattern avoidance}
\received{2021-08-11}
\revised{2022-02-28}
\accepted{2022-03-03}

\begin{document}
\publicationdetails{24}{2022}{1}{11}{8350}
\maketitle
\begin{abstract}
Motivated by the study of pattern avoidance in the context of permutations and ordered partitions, we consider the enumeration of weak-ordering chains obtained as leaves of certain restricted rooted trees. A tree of order $n$ is generated by inserting a new variable into each node at every step. A node becomes a leaf either after $n$ steps or when a certain {\em stopping condition} is met. In this paper we focus on conditions of size 2 ($x=y$, $x<y$, or $x\le y$) and several conditions of size 3. Some of the cases considered here lead to the study of descent statistics of certain `almost' pattern-avoiding permutations. 
\end{abstract}

\section{Introduction}
A \emph{weak-ordering chain} in the variables $x_1, x_2, \cdots, x_n$ is an expression of the form
\[ x_{i_1} \op x_{i_2} \op \dotsb \op x_{i_n}, \] 
where $\op$ is either $<$ or $=$. We let $\W(n)$ denote the set of all weak-ordering chains in $n$ variables. Every $w\in\W(n)$ corresponds to an ordered partition of $[n]=\{1,\dots,n\}$ obtained from the indices of the variables in $w$, where the numbers $i$ and $j$ are in the same block of the partition whenever $x_i=x_j$. For example,
\[ x_2<x_4=x_5<x_1<x_3 \;\;\longleftrightarrow\;\; \{\{2\},\{4,5\},\{1\},\{3\}\}. \] 
Therefore, if $\W(0)$ consists of the empty chain, weak-ordering chains are enumerated by the sequence $(f_n)_{n\ge 0}$ of Fubini numbers (ordered Bell numbers) $1$, $1$, $3$, $13$, $75$, $541$, $4683$, $47293,\dots$, \cite[A000670]{oeis}, which satisfy the recurrence
\[ f_0 = 1 \;\text{ and }\;  f_n = \sum_{i=1}^n \binom{n}{i} f_{n-i} \;\text{ for } n \ge 1. \]

\clearpage
Every element $w\in \W(n)$ can be recursively generated starting with $x_1$, and then inserting $x_i$ (together with either $<$ or $=$) into a previously constructed weak-ordering chain of length $i-1$. This process generates a rooted labeled tree whose nodes at level $i$ are labeled by the elements of $\W(i)$. For example, for $n=3$, we get the tree
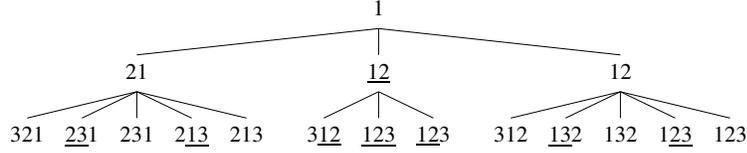
\begin{center}
\scalebox{0.85}{
\begin{forest}
for tree={s sep=15pt, inner sep=2pt, l=0}
[1
    [21, for tree={s sep=5pt} 
    	[321] [\underline{23}1] [231] [2\underline{13}] [213]
    ]
    [\underline{12}, for tree={s sep=5pt} 
	[3\underline{12}][\underline{123}][\underline{12}3]
    ]
    [12, for tree={s sep=5pt}
    	[312] [\underline{13}2] [132] [1\underline{23}] [123]
    ]
]
\end{forest}
}
\end{center}
where $ij$ is a shortcut for $x_i<x_j$ and $\underline{ij}$ represents $x_i=x_j$.

Now, suppose that we wish to stop the above generating process as soon as we have a tie. In other words, suppose that we do not allow nodes with $x_i=x_j$ for some $i>j$ to have descendants. Then, the above tree would take the form
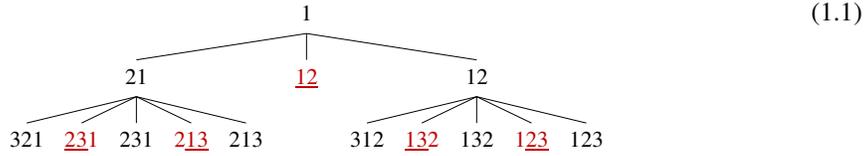
\begin{equation} \label{eq:n=3Ex}
\scalebox{0.85}{
\begin{forest}
for tree={s sep=15pt, inner sep=2pt, l=0}
[1
    [21, for tree={s sep=5pt} 
    	[321] [\col{\underline{23}1}] [231] [\col{2\underline{13}}] [213]
    ]    
    [\col{\underline{12}} [,phantom] ]
    [12, for tree={s sep=5pt}
    	[312] [\col{\underline{13}2}] [132] [\col{1\underline{23}}] [123]
    ]    
]
\end{forest}
}
\end{equation}
with only 11 leaves instead of 13. We call \eqref{eq:n=3Ex} a restricted generating tree of weak-ordering chains subject to the \emph{stopping condition} $x_i=x_j$.

As another example, consider the stopping condition $x_{i}<x_{j}<x_{k}$ with $i < j < k$. In this case, the 
generating tree at level 3 looks like the tree for $\W(3)$:
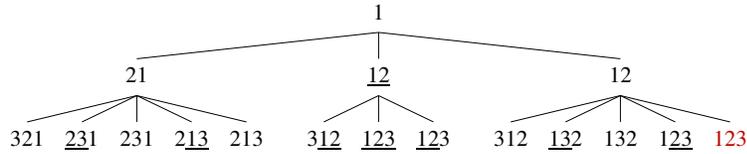
\begin{center}
\scalebox{0.85}{
\begin{forest}
for tree={s sep=15pt, inner sep=2pt, l=0}
[1
    [21, for tree={s sep=5pt} 
    	[321] [\underline{23}1] [231] [2\underline{13}] [213]
    ]
    [\underline{12}, for tree={s sep=5pt} 
	[3\underline{12}][\underline{123}][\underline{12}3]
    ]
    [12, for tree={s sep=5pt}
    	[312] [\underline{13}2] [132] [1\underline{23}] [\col{123}]
    ]
]
\end{forest}
}
\vskip6pt
\end{center}
but the node with label \col{123} will have no descendants as the generating tree grows. 

\smallskip
The goal of this paper is to study the enumeration of weak-ordering chains subject to various stopping conditions. This is equivalent to counting the number of leaves of the corresponding restricted generating subtree of $\W(n)$.

Our strategy relies on separating the leaves that avoid the stopping condition, call them \emph{active leaves}, from the leaves that contain the stopping condition, call them \emph{inactive leaves}. Throughout this paper, we will consistently use $a_n$ to denote the total number of active leaves after $n$ steps, and $b_n$ for the total number of inactive leaves. We also let $\Delta_n$ be the number of leaves that become inactive at level $n$, thus $\Delta_1=0$ and $\Delta_n=b_n-b_{n-1}$ for $n\ge 2$. Note that the meaning of active/inactive depends on the given stopping condition.

For example, for the stopping condition $x_i=x_j$, we have 
\begin{equation*}
a_1=1, \;\; b_1=0, \qquad
a_2=2, \;\; b_2=1, \qquad
a_3=6, \;\; b_3=5.
\end{equation*}
The five inactive leaves at step 3, see \eqref{eq:n=3Ex}, are the ones labelled in red.

This manuscript is organized as follows. In Section~\ref{sec:size2conditions}, we start with the simpler case when the stopping condition involves only one operation ($=$, $<$, or $\le$). In Sections~\ref{sec:123condition}--\ref{sec:213condition}, we consider stopping conditions of the form $x_{i_1} \op x_{i_2} \op x_{i_3}$ where $\op$ is either $<$ or $\le$. Our approach leads to descent statistics of certain `almost' pattern-avoiding permutations. Finally, in Section~\ref{sec:mixed}, we consider the $k$-equal case and other stopping conditions with restrictions in both the order and the sizes of the parts in the partitions corresponding to the weak-ordering chains.

\section{Trees with stopping condition of size 2}
\label{sec:size2conditions}

In this section, we will discuss the enumeration of weak-ordering chains subject to the stopping conditions $x_i=x_j$, $x_i<x_j$, and $x_i\le x_j$, respectively.

\begin{theorem}[Stopping condition $\boldsymbol{x_i = x_j}$]
If $w_n$ is the number of weak-ordering chains in $\W(n)$, subject to the stopping condition $x_i = x_j$ with $i\not= j$, then $w_n = 2n!-1$. This sequence starts with $1, 3, 11, 47, 239, 1439, 10079, 80639,\dots,$ cf.\ \cite[A020543]{oeis}.
\end{theorem}
\begin{proof}
First, since an active leaf is a weak-ordering chain that avoids a tie, the number of active leaves in $\W(n)$ is just the number of permutations on $\{x_1, \dotsc, x_n\}$, so $a_n = n!$.

On the other hand, every leaf that becomes inactive at level $j$ is a descendant of an active node at level $j-1$. In fact, everyone of these active weak-ordering chains generates $j-1$ inactive leafs, obtained by replacing $x_k$ with $x_j=x_k$ for $k\in\{1,\dots,j-1\}$. Therefore,
\[ \Delta_j = (j-1)a_{j-1} = (j-1)(j-1)!, \] 
which implies that the total number of inactive leaves is given by
\begin{equation*}
  b_{n} = \sum_{j=1}^n \Delta_j = \sum_{j=1}^n (j-1)(j-1)! =  n! -1.
\end{equation*}
In conclusion, $w_{n} = a_{n} + b_{n} = 2n! - 1$.
\end{proof}

\begin{theorem}[Stopping condition $\boldsymbol{x_i<x_j}$]
  If $w_n$ is the number of weak-ordering chains in $\W(n)$, subject to the
  stopping condition $x_i < x_j$ with $i<j$, then $w_n = (n-1)2^{n-1}+1$. That
  is, $1$, $3$, $9$, $25$, $65$, $161$, $385$, $897,\dots,$ cf.\ \cite[A002064]{oeis}.
\end{theorem}
\begin{proof}
The stopping condition implies that the indices of any active weak-ordering chain in $\W(n)$ must appear in decreasing order $x_n \op x_{n-1} \op \dotsb \op x_1$, and we can choose $\op$ to be either $<$ or $=$. Hence $a_n = 2^{n-1}$.
 
On the other hand, every active weak-ordering chain at level $j-1$ generates $j-1$ inactive leafs, obtained by replacing $x_k$ with $x_k<x_j$ for $k\in\{1,\dots,j-1\}$. Therefore,
\[ \Delta_j = (j-1)a_{j-1} = (j-1)2^{j-2}, \] 
which implies  $b_n = \sum\limits_{j=1}^n (j-1) 2^{j-2} = (n-2)2^{n-1}+1$. 

In conclusion, $w_n = a_n + b_n = 2^{n-1} + (n-2)2^{n-1} + 1 = (n-1)2^{n-1}+1$.
\end{proof}

\begin{theorem}[Stopping condition $\boldsymbol{x_i\le x_j}$]
If $w_n$ is the number of weak-ordering chains in $\W(n)$, subject to the stopping condition $x_i \le x_j$ with $i<j$, then $w_n = n^2-n+1$. These are the central polygonal numbers $1, 3, 7, 13, 21, 31, 43, 57,\dots,$ cf.\ \cite[A002061]{oeis}.
\end{theorem}
\begin{proof}
The only active chain at level $n$ is $x_n < x_{n-1} < \dotsb < x_1$, so $a_n = 1$. Moreover, the active chain $x_{j-1} < \dotsb < x_1$ generates $2(j-1)$ inactive leaves at level $j$, obtained by replacing $x_k$ with either $x_k<x_j$ or $x_k=x_j$ for $k\in\{1,\dots,j-1\}$. Therefore, $b_n = \sum\limits_{j=1}^n 2(j-1) = n(n-1)$, which implies $w_n = a_n + b_n = n^2-n+1$.
\end{proof}

\section[]{Stopping condition $\boldsymbol{x_{i_1}\!<x_{i_2}\!<x_{i_3}}$ with $\boldsymbol{i_1\!< i_2\!< i_3}$}
\label{sec:123condition}

Before we proceed, let us review some of the basic objects in the study of patterns in permutations. We refer to the book by Kitaev \cite{Kitaev11} for more details.

A permutation of size $n$ is a one-to-one function $\sigma:[n]\to[n]$. We use the common one-line notation $\sigma=\sigma(1)\cdots\sigma(n)$ and denote by $S_n$ the set of all permutations of size $n$. The {\em reverse} of $\sigma$ is the permutation $\sigma^r = \sigma(n)\cdots\sigma(1)$, and the {\em complement} is $\sigma^c = \sigma'(1)\cdots\sigma'(n)$ where $\sigma'(i)= n+1-\sigma(i)$. A permutation $\sigma$ is said to have an {\em ascent} at position $i$ if $\sigma(i)<\sigma(i+1)$, and it has a {\em descent} at $i$ if $\sigma(i)>\sigma(i+1)$.

An occurrence of a pattern $\tau$ in a permutation $\sigma$ is a subsequence in $\sigma$ (of length $|\tau|$) whose entries are in the same relative order as those in $\tau$. For example, the permutation $23154$ has two occurrences of the pattern $123$ (namely $235$ and $234$), but the permutation $53214$ avoids the pattern $123$. We use the standard notation $S_n(\tau)$ to denote the set of all permutations in $S_n$ that avoid the pattern $\tau$.

\medskip
As mentioned in the introduction, $\W(n)$ is in one-to-one correspondence with the set of ordered partitions of $[n]$. For such an ordered partition $\pi$, we let $\sigma_\pi$ be the {\em underlined permutation} obtained by merging the parts of $\pi$ and underlining the entries coming from the same block of $\pi$. In this section, we adopt the convention of writing adjacent underlined entries in decreasing order. For example,
\begin{gather*} 
x_2<x_4=x_5<x_1<x_3 \;\longleftrightarrow\; \pi = 2\,|\,54\,|\,1\,|\,3 \;\longleftrightarrow\; \sigma_\pi = 2\,\underline{54}\,13, \\
x_2=x_4=x_6<x_5<x_1=x_3 \;\longleftrightarrow\; \pi = 642\,|\,5\,|\,31 \;\longleftrightarrow\; \sigma_\pi = \underline{642}\,5\,\underline{31}. 
\end{gather*}

Let $\V_n(\sigma)$ be the set of chains in $\W(n)$ projecting to $\sigma$. A descent $\sigma(i)>\sigma(i+1)$ in the permutation $\sigma$ could come from $x_{\sigma(i)}<x_{\sigma(i+1)}$ or $x_{\sigma(i)}=x_{\sigma(i+1)}$ in the chain. Thus, if $\sigma$ has $d$ descents, then $\V_n(\sigma)$ has $2^d$ elements. Moreover, if a chain contains an increasing subsequence $x_{i_1}<x_{i_2}<x_{i_3}$, then the projected permutation must contain a $123$-pattern. As a consequence, the set of active chains in $\W(n)$, subject to the stopping condition $x_{i_1}<x_{i_2}<x_{i_3}$ with $i_1 < i_2 < i_3$, is the union
\begin{equation*}
  \bigcup_{\sigma \in S_n(123)} \V_n(\sigma)
   = \bigcup_{d=0}^{n-1} \bigcup_{\;\sigma \in S_n^d(123)} \V_n(\sigma),
\end{equation*}
where $S_n(123)$ is the set of 123-avoiding permutations on $[n]$, and 
$S_n^d(123)$ denotes the subset of permutations in $S_n(123)$ having exactly $d$
descents. Therefore, the number of active leaves at level $n$ is given by
\begin{equation}\label{eq:123active}
  a_n = \sum_{d=0}^{n-1} \sum_{\;\sigma \in S^d_n(123)} \abs{\V_n(\sigma)} = \sum_{d=0}^{n-1} 2^d e_{n,d},
\end{equation}
where $e_{n,d}=\abs{S_n^d(123)}$. Let 
\[ E(x,y)= \sum_{n=1}^{\infty}\sum_{d=0}^{n-1} e_{n,d\,} x^n y^d. \] 
Barnabei et al.\ \cite[Thm.~6]{BBS10} gave a closed form for $1+E(x,y)$ from which we deduce
\begin{equation}\label{eq:E(x,y)}
  E(x,y) = \frac{1  - 2xy(1+x-xy) - \sqrt{1 - 4xy(1+x-xy)}\ }{2xy^2(1+x-xy)}.
\end{equation}
Therefore, the generating function $A(x) = \sum\limits_{n=1}^{\infty} a_n x^n$ satisfies
\begin{equation*}
 A(x) = E(x,2) = \frac{(1-2x)^2 - \sqrt{1 - 8x + 8x^2}\ }{8x(1-x)}.
\end{equation*}
This result is consistent with Chen et al.\ \cite[Cor.~2.3]{CDZ14}. Moreover, since $A(x)$ can be expressed in terms of the Catalan generating function, namely $A(x) = \frac12 C(2x(1-x))$, we have
\begin{equation}\label{eq:123active_closedForm}
 a_n = \sum_{j=0}^n (-1)^{j}2^{n-j-1} \binom{n-j}{j} C_{n-j} \;\text{ for } n\ge 1,
\end{equation}
where $C_m$ denotes the $m$-th Catalan number. This sequence starts with $1$, $3$, $12$, $56$, $284$, $1516$, $8384$, $47600,\dots,$ cf.\ \cite[A226316]{oeis}.

Using a version of Krattenthaler's bijection between Dyck paths and $123$-avoiding permutations (cf.\ \cite{Kitaev11,Kra01}), it can be checked that the set of active chains in $\W(n)$ with stopping condition $x_{i_1}<x_{i_2}<x_{i_3}$ is in one-to-one correspondence with the set of Dyck paths of semilength $n$ where valleys and triple down-steps come in 2 colors.\footnote{A Dyck path of semilength $n$ is a lattice path from $(0,0)$ to $(0,2n)$ with steps $\mathsf{U} = (1, 1)$ and $\mathsf{D} = (1, -1)$, never going below the $x$-axis. A {\em valley} is a subpath $\mathsf{DU}$.}
 
\bigskip
Our connection between weak-ordering chains and permutations makes it clear that a leaf is inactive if the associated permutation has a $123$ pattern. Thus, in order to count the elements that become inactive at level $n$, we need to enumerate the following set:

For $n>3$ and $1\le d\le n-3$, define
\begin{equation}\label{eq:just123}
 \G^{d}_{n}(123) = \{\sigma\in S_n \,|\, \sigma \text{ has a $123$ pattern, $d$ descents, and } \sigma'\in S_{n-1}(123)\},
\end{equation}
where $\sigma'\in S_{n-1}$ denotes the permutation obtained from $\sigma\in S_n$ by removing $n$. In addition, we define $\G^{d}_{1}(123)=\G^{d}_{2}(123)=\varnothing$ and $\G^{0}_{3}(123)=\{123\}$. 

\begin{proposition} \label{prop:gnd}
If $g_{n,d}=\abs{\G^{d}_{n}(123)}$, then $g_{3,0}=1$, and for $n>3$ and $1\le d\le n-3$,
\begin{equation}\label{eq:gnd}
 g_{n,d} = (d+1)e_{n-1,d} + (n-d) e_{n-1,d-1} - e_{n,d},
\end{equation}
where $e_{n,d}=\abs{S_n^d(123)}$. Therefore, $G(x,y)= x^3 + \sum\limits_{n=4}^{\infty}\sum\limits_{d=1}^{n-3} g_{n,d\,} x^n y^d$ satisfies
\begin{equation*}
 G(x,y) = x+(x-1)E(x,y)+x^2yE_x(x,y)+(xy-xy^2)E_y(x,y).
\end{equation*}
where $\langle E_x,E_y\rangle = \nabla E$, and $E(x,y)$ is the function in \eqref{eq:E(x,y)}. In other words,
\begin{equation*}
 G(x,y) = \frac{1-4xy(1+x-xy) +2x^2y - (1-2xy) \sqrt{1-4xy(1+x-xy)}}{2xy^2\sqrt{1-4xy(1+x-xy)}}.
\end{equation*}
\end{proposition}
\begin{proof}
For $n>3$, every $\sigma$ in $\G^{d}_{n}(123)$ can be generated by inserting $n$ into a permutation $\sigma'\in S^d_{n-1}(123) \cup S^{d-1}_{n-1}(123)$ at a position where it creates a 123 pattern.

If $\sigma'$ has $d$ descents, then $n$ may only be inserted at a descent or at the last position of $\sigma'$. Thus, each $\sigma'\in S^{d}_{n-1}(123)$ generates $d+1$ permutations in $S^{d}_{n}$. On the other hand, if $\sigma'$ has $d-1$ descents, then $n$ will have to be inserted at one of the $n-d$ available ascents of $\sigma'$ in order to create an extra descent. Therefore, each $\sigma'\in S^{d-1}_{n-1}(123)$ generates $n-d$ permutations in $S^{d}_{n}$. 

Together, the above insertion procedures generate $(d+1)e_{n-1,d} + (n-d) e_{n-1,d-1}$ permutations of size $n$ with $d$ descents. However, since the elements of $\G^{d}_{n}(123)$ are required to have a 123 pattern, we need to remove the $e_{n,d}$ permutations that avoid 123.

The first statement about $G(x,y)$ then follows from \eqref{eq:gnd} by means of routine algebraic manipulations. The closed form statement is just a consequence of \eqref{eq:E(x,y)}.
\end{proof}

\begin{theorem} \label{thm:x<y<z}
If $w_n$ is the number of weak-ordering chains in $\W(n)$, subject to the stopping condition $x_{i_1}<x_{i_2}<x_{i_3}$ with $i_1 < i_2 < i_3$, then
\begin{equation*}
  w_n = \sum_{d=0}^{n-1}2^d e_{n,d} + \sum_{j=3}^n\sum_{d=0}^{j-3}2^d g_{j,d}.
\end{equation*}
The sequence $(w_n)_{n\in\mathbb{N}}$ starts with $1, 3, 13, 69, 401, 2433, 15121, 95441,\dots,$ and $W(x) = \sum\limits_{n=1}^\infty w_n x^n$ satisfies
\[ W(x) = \frac{x + x \sqrt{1 - 8x + 8x^2}}{2(1-x) \sqrt{1 - 8x + 8x^2}}. \]
\end{theorem}
\begin{proof}
By \eqref{eq:123active}, the number of active leaves after $n$ steps is $a_n=\sum\limits_{d=0}^{n-1}2^d e_{n,d}$, and their generating function $A(x)$ is given by
\[ A(x) = E(x,2) = \frac{(1-2x)^2 - \sqrt{1 - 8x + 8x^2}\ }{8x(1-x)}. \]
On the other hand, a weak-ordering chain that becomes inactive at level $j\ge 3$ leads to a $\sigma\in S_j$ that has a $123$ pattern, and such that $\sigma'\in S_{j-1}(123)$. The number of such permutations having $d$ descents is given by $g_{j,d}$ in \eqref{eq:gnd}. Now, since a descent $\sigma(i)>\sigma(i+1)$ in $\sigma$ comes from either $x_{\sigma(i)}<x_{\sigma(i+1)}$ or $x_{\sigma(i)}=x_{\sigma(i+1)}$ in the chain, there are $2^d g_{j,d}$ chains projecting to the permutations in $\G_j^d(123)$. Therefore,
\[ b_1=b_2=0, \text{ and }\; \Delta_j = b_j-b_{j-1} = \sum_{d=0}^{j-3}2^d g_{j,d} \text{ for } j\ge 3. \]
The formula for $w_n= a_n + b_n$ follows from the fact that $b_n=\sum\limits_{j=3}^n \Delta_j$.

Finally, using Proposition~\ref{prop:gnd} we can write $B(x)=\sum\limits_{n=1}^\infty b_n x^n$ as
\[ B(x) = \frac{G(x,2)}{1-x} = \frac{1-8x+12x^2 - (1-4x) \sqrt{1-8x+8x^2}}{8x(1-x)\sqrt{1-8x+8x^2}}, \]
and the claimed expression for $W(x)$ is obtained by combining $A(x) + B(x)$.
\end{proof}

\section[]{Stopping condition  $\boldsymbol{x_{i_1}\!\le x_{i_2}\!\le x_{i_3}}$ with $\boldsymbol{i_1\!< i_2\!< i_3}$}
\label{sec:condition}

This case is more restrictive than the previous one. For example, the chains
\begin{equation*}
 x_1=x_2<x_3, \;\; x_1<x_2=x_3, \,\text{ and }\, x_1=x_2=x_3, 
\end{equation*}
that were previously active leaves, will now turn into inactive leaves.

In order to handle the stopping condition at the permutation level in a consistent manner (i.e.\ using descents in the permutation to mark the places where the corresponding chain may have an $=$ symbol), it is convenient to consider underlined $321$-avoiding permutations. We will then use the complement map to obtain underlined $123$-avoiding permutations that correspond to the weak-ordering chains with the stopping condition $x_{i_1}\le x_{i_2}\le x_{i_3}$. 

This process will be explained and illustrated in the proof of the following proposition.

\begin{proposition} \label{prop:321active}
The set of active leaves in $\W(n)$, subject to the stopping condition $x_{i_1}\le x_{i_2}\le x_{i_3}$ with $i_1 < i_2 < i_3$, is in one-to-one correspondence with the set of Dyck paths of semilength $n$ where each subpath {\sf UDD} can take on two colors. If we let $a_n$ denote the number of such paths, then their generating function $A(x)$ satisfies
\[ 1+ A(x) = \frac{1 - \sqrt{1 - 4x - 4x^2}}{2x(1+x)} = C(x(1+x)), \]
where $C(x)$ is the Catalan function, and therefore, $a_n = \sum\limits_{j=0}^n \binom{n-j}{j} C_{n-j}$. This gives the sequence $1$, $3$, $9$, $31$, $113$, $431$, $1697$, $6847,\dots,$ cf.\ \cite[A052709]{oeis}.
\end{proposition}
\begin{proof}
We will prove the statement by establishing a bijection that relies on a known map between Dyck paths and 321-avoiding permutations, see e.g.\ \cite{Kitaev11}. We represent a Dyck path as a lattice path from $(0,0)$ to $(n,n)$, starting with an \textsf{N}-step $(0,1)$, ending with an $\textsf{E}$-step $(1,0)$, and never going below the line $y=x$. We will allow \textsf{NEE} to take on two colors.

Given a Dyck path $P$ of the above type, drawn on a coordinate grid, place a dot in every cell bounded by an \textsf{NE}-turn of the path, and then place dots in increasing order (from left to right) so that every column below the path has exactly one dot with no two dots in the same row. This gives the plot of a 321-avoiding permutation $\sigma_P$, and we underline adjacent elements of the permutation if the corresponding steps in the path are orange. For example, for $n=3$ we have the following 9 elements:

\smallskip
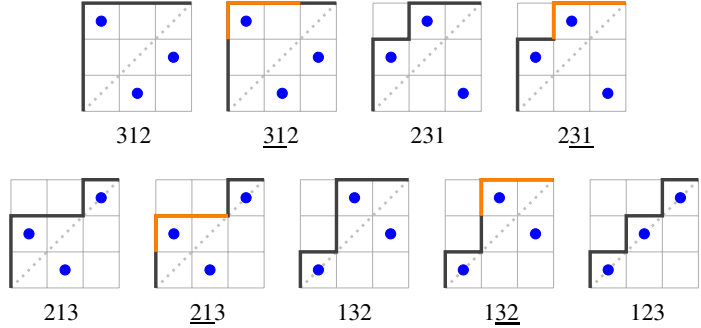
\begin{center}
\def\sf{0.48}
\def\rr{0.08}
\def\Rr{0.14}
\tikz[scale=\sf]{
 \draw[step=1,gray!60] (0,0) grid (3,3); \draw[pathlight] (0,0) -- (3,3);
 \draw[path] (0,0) -- (0,3) -- (3,3);
 \foreach \x/\y in {1/3,2/1,3/2} {\draw[fill,blue] (\x-0.5,\y-0.5) circle(\Rr);}
}
\quad
\tikz[scale=\sf]{
 \draw[step=1,gray!60] (0,0) grid (3,3); \draw[pathlight] (0,0) -- (3,3);
 \draw[path] (0,0) -- (0,3) -- (3,3);
 \draw[path,orange] (0,2) -- (0,3) -- (2,3);
 \foreach \x/\y in {1/3,2/1,3/2} {\draw[fill,blue] (\x-0.5,\y-0.5) circle(\Rr);}
}
\quad
\tikz[scale=\sf]{
 \draw[step=1,gray!60] (0,0) grid (3,3); \draw[pathlight] (0,0) -- (3,3);
 \draw[path] (0,0) -- (0,2) -- (1,2) -- (1,3) -- (3,3);
 \foreach \x/\y in {1/2,2/3,3/1} {\draw[fill,blue] (\x-0.5,\y-0.5) circle(\Rr);}
}
\quad
\tikz[scale=\sf]{
 \draw[step=1,gray!60] (0,0) grid (3,3); \draw[pathlight] (0,0) -- (3,3);
 \draw[path] (0,0) -- (0,2) -- (1,2) -- (1,3) -- (3,3);
 \draw[path,orange] (1,2) -- (1,3) -- (3,3);
 \foreach \x/\y in {1/2,2/3,3/1} {\draw[fill,blue] (\x-0.5,\y-0.5) circle(\Rr);}
}

\parbox{0.43\textwidth}{\small 312 \hfill \underline{31}2 \hfill 231 \hfill 2\underline{31}}

\vspace{10pt}
\tikz[scale=\sf]{
 \draw[step=1,gray!60] (0,0) grid (3,3); \draw[pathlight] (0,0) -- (3,3);
 \draw[path] (0,0) -- (0,2) -- (2,2) -- (2,3) -- (3,3);
 \foreach \x/\y in {1/2,2/1,3/3} {\draw[fill,blue] (\x-0.5,\y-0.5) circle(\Rr);}
}
\quad
\tikz[scale=\sf]{
 \draw[step=1,gray!60] (0,0) grid (3,3); \draw[pathlight] (0,0) -- (3,3);
 \draw[path] (0,0) -- (0,2) -- (2,2) -- (2,3) -- (3,3);
 \draw[path,orange] (0,1) -- (0,2) -- (2,2);
 \foreach \x/\y in {1/2,2/1,3/3} {\draw[fill,blue] (\x-0.5,\y-0.5) circle(\Rr);}
}
\quad
\tikz[scale=\sf]{
 \draw[step=1,gray!60] (0,0) grid (3,3); \draw[pathlight] (0,0) -- (3,3);
 \draw[path] (0,0) -- (0,1) -- (1,1) -- (1,3) -- (3,3);
 \foreach \x/\y in {1/1,2/3,3/2} {\draw[fill,blue] (\x-0.5,\y-0.5) circle(\Rr);}
}
\quad
\tikz[scale=\sf]{
 \draw[step=1,gray!60] (0,0) grid (3,3); \draw[pathlight] (0,0) -- (3,3);
 \draw[path] (0,0) -- (0,1) -- (1,1) -- (1,3) -- (3,3);
 \draw[path,orange] (1,2) -- (1,3) -- (3,3);
 \foreach \x/\y in {1/1,2/3,3/2} {\draw[fill,blue] (\x-0.5,\y-0.5) circle(\Rr);}
}
\quad
\tikz[scale=\sf]{
 \draw[step=1,gray!60] (0,0) grid (3,3); \draw[pathlight] (0,0) -- (3,3);
 \draw[path] (0,0) -- (0,1) -- (1,1) -- (1,2) -- (2,2) -- (2,3) -- (3,3);
 \foreach \x/\y in {1/1,2/2,3/3} {\draw[fill,blue] (\x-0.5,\y-0.5) circle(\Rr);}
}
\parbox{0.56\textwidth}{\small
 213 \hfill \underline{21}3 \hfill 132 \hfill 1\underline{32} \hfill 123}
\end{center}

Finally, we take the complement of $\sigma_P$ (keeping the underlines) and construct the weak-ordering chain $w_P$ associated with $P$ by replacing $ij$ in $\sigma_P^c$ with $x_i<x_j$ and $\underline{ij}$ with $x_i=x_j$. Note that $\sigma_P^c\in S_n(123)$, hence $w_P$ does not contain the stopping condition. Also observe that, while $x_i=x_j$ and $x_j=x_i$ represent the same chain, our map gives the partition blocks in the form: smallest element first, followed by the rest of the elements in decreasing order.

The above paths correspond to the chains:
\begin{gather*}
 x_1<x_3<x_2, \;\; x_1=x_3<x_2, \;\; x_2<x_1<x_3, \;\; x_2<x_1=x_3 \\
 x_2<x_3<x_1, \;\; x_2=x_3<x_1, \;\; x_3<x_1<x_2, \;\; x_3<x_1=x_2, \;\; x_3<x_2<x_1.
\end{gather*}

We now illustrate the inverse by means of an example. Suppose we have the active leaf
\[ w:\quad x_5=x_7<x_2=x_6<x_4<x_1<x_3. \]
To this weak-ordering chain we associate the underlined permutation $\sigma_w = \underline{57}\,\underline{26}\,413$ with complement $\sigma_w^c = \underline{31}\,\underline{62}\,475$. Finally, constructing the Dyck path corresponding to $3162475$ and making orange the \textsf{NEE} steps above underlined numbers, we arrive at 

\smallskip
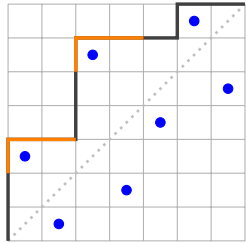
\begin{center}
\def\rr{0.08}
\def\Rr{0.14}
\tikz[scale=0.45]{
 \draw[step=1,gray!60] (0,0) grid (7,7); \draw[pathlight] (0,0) -- (7,7);
 \draw[path] (0,0) -- (0,3) -- (2,3) -- (2,6) -- (5,6) -- (5,7) -- (7,7);
 \draw[path,orange] (0,2) -- (0,3) -- (2,3);
 \draw[path,orange] (2,5) -- (2,6) -- (4,6);
 \foreach \x/\y in {1/3,2/1,3/6,4/2,5/4,6/7,7/5} {\draw[fill,blue] (\x-0.5,\y-0.5) circle(\Rr);}
}
\end{center}
which is a colored Dyck path of the desired form.

Note that a permutation $\sigma$ on $[n]$ with $d$ descents gives rise to $2^d$ chains. Since the reverse $\sigma^r$ is a 123-avoiding permutation with $d$ ascents (hence $n-1-d$ descents), we conclude that the number of active leaves is given by
\begin{equation}\label{eq:321actives}
 a_n = \sum_{d=0}^{n-1} 2^{d} e_{n,n-1-d} = \sum_{d=0}^{n-1} 2^{n-1-d} e_{n,d},
\end{equation}
where $e_{n,d}=\abs{S_n^d(123)}$. Thus $A(x) = \tfrac12 E(2x,\tfrac12)$ with $E(x,y)$ from \eqref{eq:E(x,y)}, which simplifies to the claimed expression.
\end{proof}

\begin{theorem} \label{thm:x<=y<=z}
If $w_n$ is the number of weak-ordering chains in $\W(n)$, subject to the stopping condition $x_{i_1}\le x_{i_2}\le x_{i_3}$ with $i_1 < i_2 < i_3$, then
\begin{equation*}
  w_n = \sum_{d=0}^{n-1}2^{n-1-d} e_{n,d} + \sum_{j=3}^n\sum_{d=2}^{j-1}2^{j-1-d} g_{j,d}.
\end{equation*}
The sequence $(w_n)_{n\in\mathbb{N}}$ starts with $1, 3, 13, 59, 269,1227, 5613, 25771, 118765,\dots,$ and the generating function $W(x) = \sum\limits_{n=1}^\infty w_n x^n$ satisfies
\[ 1+ W(x) = \frac{x}{1-x^2} + \frac{1-2x-2x^{2}}{(1-x^2)\sqrt{1-4x-4x^2}}. \]
\end{theorem}
\begin{proof}
We already know that the number of active chains is given by \eqref{eq:321actives}, so we only need to focus on the inactive ones. Note that inactive chains at level $j$ come from permutations in $S_j$ that contain the pattern 321 and such that their reduced permutation (obtained by removing $j$) belongs to $S_{j-1}(321)$. The set of such permutations having exactly $d$ descents is in bijection with $\G^{j-1-d}_{j}(123)$ (as defined in \eqref{eq:just123}), and therefore, each such permutation induces $2^d g_{j,j-1-d}$ inactive weak-ordering chains. Hence
\[ \Delta_j = b_j-b_{j-1} =  \sum_{d=0}^{j-3} 2^{d} g_{j,j-1-d} = \sum_{d=2}^{j-1} 2^{j-1-d} g_{j,d}, \]
which implies $b_n = \sum\limits_{j=3}^{n} \Delta_j = \sum\limits_{j=3}^{n} \sum\limits_{d=2}^{j-1} 2^{j-1-d} g_{j,d}$. We can then use Proposition~\ref{prop:gnd} to write its generating function $B(x)$ as
\[ B(x) = \frac{\frac12 G(2x,\tfrac12)}{1-x} = \frac{1-4x - (1-2x) \sqrt{1-4x-4x^2}}{2x(1-x)\sqrt{1-4x-4x^2}}. \]
Combined with Proposition~\ref{prop:321active}, this gives a closed form for $1+W(x)=1+A(x)+B(x)$.
\end{proof}

\section[]{Stopping condition $\boldsymbol{x_{i_1}\!\le x_{i_2}\!<x_{i_3}}$ with $\boldsymbol{i_1\!< i_2\!< i_3}$}
\label{sec:213condition}

\begin{proposition} \label{prop:213active}
The set of active leaves in $\W(n)$, subject to the stopping condition $x_{i_1}\le x_{i_2}<x_{i_3}$ with $i_1 < i_2 < i_3$, is in one-to-one correspondence with the set of Dyck paths of semilength $n$ in which valleys may be marked. They are counted by the little Schr\"oder numbers $1$, $3$, $11$, $45,\dots,$ \cite[A001003]{oeis}. 
\end{proposition}
\begin{proof}
We will use the same map between 321-avoiding permutations and Dyck paths used in Proposition~\ref{prop:321active} to provide a bijection between active weak-ordering chains of length $n$ and Dyck paths from $(0, 0)$ to $(n, n)$ in which valleys may be marked. It is easy to check that these paths are counted by the little Schr\"oder numbers.

For every possibly marked Dyck path $P$, we plot the corresponding 321-avoiding permutation $\sigma_P$ and underline adjacent elements of $\sigma_P$ if there is a marked valley between their plots. For example, for $n=3$ there are 11 such elements:

\smallskip
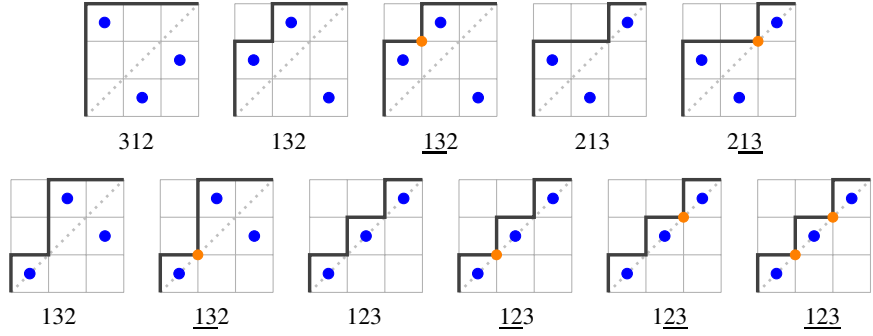
\begin{center}
\def\sf{0.5}
\def\rr{0.13}
\def\Rr{0.14}

\tikz[scale=\sf]{
 \draw[step=1,gray!60] (0,0) grid (3,3); \draw[pathlight] (0,0) -- (3,3);
 \draw[path] (0,0) -- (0,3) -- (3,3);
 \foreach \x/\y in {1/3,2/1,3/2} {\draw[fill,blue] (\x-0.5,\y-0.5) circle(\Rr);}
}
\quad
\tikz[scale=\sf]{
 \draw[step=1,gray!60] (0,0) grid (3,3); \draw[pathlight] (0,0) -- (3,3);
 \draw[path] (0,0) -- (0,2) -- (1,2) -- (1,3) -- (3,3);
 \foreach \x/\y in {1/2,2/3,3/1} {\draw[fill,blue] (\x-0.5,\y-0.5) circle(\Rr);}
}
\quad
\tikz[scale=\sf]{
 \draw[step=1,gray!60] (0,0) grid (3,3); \draw[pathlight] (0,0) -- (3,3);
 \draw[path] (0,0) -- (0,2) -- (1,2) -- (1,3) -- (3,3);
 \draw[fill,orange] (1,2) circle(\rr);
 \foreach \x/\y in {1/2,2/3,3/1} {\draw[fill,blue] (\x-0.5,\y-0.5) circle(\Rr);}
}
\quad
\tikz[scale=\sf]{
 \draw[step=1,gray!60] (0,0) grid (3,3); \draw[pathlight] (0,0) -- (3,3);
 \draw[path] (0,0) -- (0,2) -- (2,2) -- (2,3) -- (3,3);
 \foreach \x/\y in {1/2,2/1,3/3} {\draw[fill,blue] (\x-0.5,\y-0.5) circle(\Rr);}
}
\quad
\tikz[scale=\sf]{
 \draw[step=1,gray!60] (0,0) grid (3,3); \draw[pathlight] (0,0) -- (3,3);
 \draw[path] (0,0) -- (0,2) -- (2,2) -- (2,3) -- (3,3);
 \draw[fill,orange] (2,2) circle(\rr);
 \foreach \x/\y in {1/2,2/1,3/3} {\draw[fill,blue] (\x-0.5,\y-0.5) circle(\Rr);}
}
\parbox{0.58\textwidth}{\small
 312 \hfill 132 \hfill \underline{13}2 \hfill 213 \hfill 2\underline{13}}
 
\vspace{8pt}
\tikz[scale=\sf]{
 \draw[step=1,gray!60] (0,0) grid (3,3); \draw[pathlight] (0,0) -- (3,3);
 \draw[path] (0,0) -- (0,1) -- (1,1) -- (1,3) -- (3,3);
 \foreach \x/\y in {1/1,2/3,3/2} {\draw[fill,blue] (\x-0.5,\y-0.5) circle(\Rr);}
}
\quad
\tikz[scale=\sf]{
 \draw[step=1,gray!60] (0,0) grid (3,3); \draw[pathlight] (0,0) -- (3,3);
 \draw[path] (0,0) -- (0,1) -- (1,1) -- (1,3) -- (3,3);
 \draw[fill,orange] (1,1) circle(\rr);
 \foreach \x/\y in {1/1,2/3,3/2} {\draw[fill,blue] (\x-0.5,\y-0.5) circle(\Rr);}
}
\quad
\tikz[scale=\sf]{
 \draw[step=1,gray!60] (0,0) grid (3,3); \draw[pathlight] (0,0) -- (3,3);
 \draw[path] (0,0) -- (0,1) -- (1,1) -- (1,2) -- (2,2) -- (2,3) -- (3,3);
 \foreach \x/\y in {1/1,2/2,3/3} {\draw[fill,blue] (\x-0.5,\y-0.5) circle(\Rr);}
}
\quad
\tikz[scale=\sf]{
 \draw[step=1,gray!60] (0,0) grid (3,3); \draw[pathlight] (0,0) -- (3,3);
 \draw[path] (0,0) -- (0,1) -- (1,1) -- (1,2) -- (2,2) -- (2,3) -- (3,3);
 \draw[fill,orange] (1,1) circle(\rr);
 \foreach \x/\y in {1/1,2/2,3/3} {\draw[fill,blue] (\x-0.5,\y-0.5) circle(\Rr);}
}
\quad
\tikz[scale=\sf]{
 \draw[step=1,gray!60] (0,0) grid (3,3); \draw[pathlight] (0,0) -- (3,3);
 \draw[path] (0,0) -- (0,1) -- (1,1) -- (1,2) -- (2,2) -- (2,3) -- (3,3);
 \draw[fill,orange] (2,2) circle(\rr);
 \foreach \x/\y in {1/1,2/2,3/3} {\draw[fill,blue] (\x-0.5,\y-0.5) circle(\Rr);}
}
\quad
\tikz[scale=\sf]{
 \draw[step=1,gray!60] (0,0) grid (3,3); \draw[pathlight] (0,0) -- (3,3);
 \draw[path] (0,0) -- (0,1) -- (1,1) -- (1,2) -- (2,2) -- (2,3) -- (3,3);
 \draw[fill,orange] (1,1) circle(\rr); \draw[fill,orange] (2,2) circle(\rr);
 \foreach \x/\y in {1/1,2/2,3/3} {\draw[fill,blue] (\x-0.5,\y-0.5) circle(\Rr);}
}
\parbox{0.72\textwidth}{\small
 132 \hfill \underline{13}2 \hfill 123 \hfill \underline{12}3 \hfill 1\underline{23} \hfill \underline{123}}
\end{center}

Note that, by construction, a 3\underline{12} pattern can never occur. We then take the reverse of $\sigma_P$ (keeping the underlines) and construct the weak-ordering chain $w_P$ associated with $P$ by replacing $ij$ in $\sigma_P^r$ with $x_i<x_j$ and $\underline{ij}$ with $x_i=x_j$. The resulting underlined permutation $\sigma_P^r$ avoids 123 and does not contain a \underline{21}3 pattern. Hence, the chain $w_P$ does not contain the stopping condition. For example, the above paths correspond to the active chains:
\begin{gather*}
 x_2<x_1<x_3, \;\;  x_2<x_3<x_1, \;\; x_2<x_3=x_1, \;\; x_3<x_1<x_2, \;\; x_3=x_1<x_2, \\
 x_2<x_3<x_1, \;\, x_2<x_3=x_1, \;\, x_3<x_2<x_1, \;\, x_3<x_2=x_1, \;\, x_3=x_2<x_1, \; x_3=x_2=x_1.
\end{gather*}

The reverse map is straightforward. From an active chain, read the underlined permutation made from its indices (with the convention that in a block of equal elements, indices are sorted in decreasing order). Then take the reverse permutation (necessarily 321-avoiding) and draw the associated Dyck paths, marking the valleys that correspond to underlined adjacent entries in the permutation.
\end{proof}

Observe that, since the number of Dyck paths with $v$ valleys is counted by the Narayana numbers $N_{n,v}  = \frac{1}{n} \binom{n}{v}\binom{n}{v+1}$, the number of active leaves is given by
\[ a_n = \sum_{v=0}^{n-1} 2^v N_{n,v}  = \sum_{v=0}^{n-1} \,\frac{2^v}{n} \binom{n}{v}\binom{n}{v+1}. \]

\clearpage
For $n>3$ and $1\le d\le n-2$, define
\begin{equation*}
 \G^{d}_{n}(213) = \{\sigma\in S_n \,|\, \sigma \text{ has a $213$ pattern, $d$ descents, and } \sigma'\in S_{n-1}(213)\},
\end{equation*}
where $\sigma'\in S_{n-1}$ is the reduced permutation obtained by removing $n$. In addition, we define $\G^{d}_{1}(213)=\G^{d}_{2}(213)=\varnothing$ and $\G^{1}_{3}(213)=\{213\}$. 

\begin{proposition} \label{prop:lnd}
If $\ell_{n,d}=\abs{\G^{d}_{n}(213)}$, then $\ell_{3,1}=1$, and for $n>3$ and $1\le d\le n-2$,
\begin{equation*}
 \ell_{n,d} = (d+1) N_{n-1,d} + (n-d) N_{n-1,d-1} - N_{n,d}.
\end{equation*}
Moreover, $L(x,y) = x^3 + \sum\limits_{n=4}^\infty\sum\limits_{d=1}^{n-2} \ell_{n,d\,} x^n y^d$ has the closed form
\begin{equation*}
 L(x,y) = \frac{x(y+1) - 2(y^2 - y)x^4 - 1}{2xy} + \frac{(1-xy)^2+x^2-2x}{2xy\sqrt{1-2x(y+1)+x^2(y-1)^2}}.
\end{equation*}
\end{proposition}
\begin{proof}
The formula for $\ell_{n,d}$ follows from an argument similar to the one made for the proof of \eqref{eq:gnd}. There are $(d+1)N_{n-1,d}$ permutations of size $n$ with $d$ descents that can be created by inserting $n$ (at a descent or at the end) into the $N_{n-1,d}$ elements of $S^d_{n-1}(213)$. In addition, there are $(n-d)N_{n-1,d-1}$ such permutations that can be generated from $S^{d-1}_{n-1}(213)$. To get the total number of elements in $\G^{d}_{n}(213)$, we combine the above permutations and remove the ones that avoid 213 (counted by $N_{n,d}$).

As a consequence, if $N(x,y)$ is the generating function for the Narayana numbers, then
\begin{equation*}
 L(x,y) = x + (1-y)x^3 + (x-1)N(x,y)+x^2yN_x(x,y)+(xy-xy^2)N_y(x,y).
\end{equation*}
Finally, using the known formula
\[ N(x,y) = \frac{1-x(y+1)-\sqrt{1-2x(y+1)+x^2(y-1)^2}}{2xy}, \]
one derives the closed form of $L(x,y)$.
\end{proof}

\begin{theorem} \label{thm:x<=y<z}
If $w_n$ is the number of weak-ordering chains in $\W(n)$, subject to the stopping condition $x_{i_1}\le x_{i_2} < x_{i_3}$ with $i_1 < i_2 < i_3$, then
\begin{equation*}
  w_n = \sum_{d=0}^{n-1}2^{d} N_{n,d} + \sum_{j=3}^n\sum_{d=1}^{j-2}2^{d} \ell_{j,d}.
\end{equation*}
The sequence $(w_n)_{n\in\mathbb{N}}$ starts with $1$, $3$, $13$, $65$, $341$, $1827$, $9913$, $54273$, $299209$, $1658723,\dots,$ and the function $W(x) = \sum\limits_{n=1}^\infty w_n x^n$ satisfies
\[ W(x) = \frac{(1-x)^2 - (1-3x)\sqrt{1- 6x + x^2}}{4(1-x)\sqrt{1- 6x + x^2}}. \]
\end{theorem}
\begin{proof}
The first summation in the claimed formula for $w_n$ represents the number of active weak-ordering chains (Proposition~\ref{prop:213active}). Thus their generating function can be written as
\[ A(x) = N(x,2) = \frac{1-3x-\sqrt{1-6x+x^2}}{4x}. \]
In order to derive a formula for the number $\Delta_j$ of active chains that become inactive at level $j$, we will provide a bijection $\phi$ between these chains and the set of permutations in $\bigcup_{d=1}^{j-2}\G^{d}_{j}(213)$ where descents may be underlined. That implies $\Delta_1=\Delta_2=0$, and for $j\ge 3$, $\Delta_j = \sum\limits_{d=1}^{j-2}2^d \ell_{j,d}$. As a consequence, the total number of inactive chains will be given by
\[ b_1=b_2=0, \text{ and } b_n = \sum_{j=3}^{n} \Delta_j = \sum_{j=3}^n\sum_{d=1}^{j-2}2^{d} \ell_{j,d} \text{ for } n\ge 3, \]
with generating function
\[ B(x) = \frac{x^3+L(x,2)}{1-x} = \frac{3x-1}{4x(1-x)} - \frac{5x-1}{4x\sqrt{1-6x+x^2}}. \]
Combining the active and inactive chains, we get the formulas for $w_n$ and $W(x)$.

We finish the proof by describing the bijection $\phi$. Let $w\in \W(n)$ be a chain that becomes inactive at level $n$, and let $\sigma_w$ be the permutation corresponding to $w$ obtained from the indices of the variables, with the usual convention that indices of equal elements are underlined and listed in decreasing order. Observe that for $w$ to be inactive, the entry $n$ in $\sigma_w$ must be part of either a $123$ pattern or a $\underline{21}3$ pattern, while the reduced permutation $\sigma_w'$ must avoid both patterns. In particular, this implies that the entry $n-1$ in $\sigma_w$ cannot be to the right of entry $n$.

If $n-1$ and $n$ are not adjacent, we proceed as follows:
\begin{itemize}
\item Map $\sigma_w'$ to the 213-avoiding permutation $\tau_w'$ having the same right-to-left maxima\footnote{A right-to-left maximum of a permutation is an entry with no larger entries to its right.}, keeping the underlines in the same positions.
\item If $i$ is the position of $n$ in $\sigma_w$, we let $\phi(w)$ be the permutation obtained by inserting $n$ into $\tau_w'$ at position $i$.
\end{itemize}
Since $n-1$ must be left of $n$ in $\sigma_w$, the permutation $\phi(w)$ will always contain a 213-pattern. For example, for the chain $x_6<x_8<x_7=x_4<x_2<x_9<x_5=x_1<x_3$, the above steps give the permutations
\begin{gather*}
 \sigma_w = 68\underline{74}29\underline{51}3 \\
 \sigma_w' = 6\col{8}\underline{\col{7}4}2\underline{\col{5}1}\col{3} \mapsto \tau_w' = 6\col{8}\underline{\col{7}1}4\underline{\col{5}2}\col{3} \\
 \phi(w) = 68\underline{71}49\underline{52}3.
\end{gather*}

\clearpage
If $n-1$ and $n$ are adjacent in $\sigma_w$, then it must be of the form $\sigma_0(n-1)n*$, where $\sigma_0$ is a decreasing permutation of size at least 1 with no underlines. In this case, we create the permutation $\tilde\sigma_w = (n-1)\sigma_0n*$ and proceed as in the previous case. For example, for $x_6<x_5<x_8<x_9<x_7=x_4=x_2<x_1<x_3$, we get:
\begin{gather*}
 \sigma_w = 6589\underline{742}13 \mapsto \tilde\sigma_w = 8659\underline{742}13 \\
 \tilde\sigma_w' = \col{8}65\underline{\col{74}2}1\col{3} \mapsto \tau_w' =  \col{8}56\underline{\col{74}1}2\col{3}\\
 \phi(w) = 8569\underline{741}23.
\end{gather*}
Observe that if the original permutation $\sigma_w$ starts with $n-1$, there must be either an ascent or a $\underline{21}$ pattern left of $n$, so $\tilde\sigma_w$ cannot create a duplicate. The above algorithm can be easily reversed, showing the invertibility of $\phi$.
\end{proof}

\section{Other stopping conditions}
\label{sec:mixed}

In this section, we generalize the stopping condition $x_i=x_j$ to chains of arbitrary length. We also consider some conditions with restrictions in both the order and the sizes of the parts in the partitions corresponding to the weak-ordering chains.

\subsection*{The $k$-equal stopping condition}
For the stopping condition $x_{i_1} = \cdots =x_{i_k}$ with $k>1$, active chains in the set $\W(n)$ correspond to ordered partitions of $[n]$ with parts of size at most $k-1$. Thus, $a_n$ (number of active leaves at level $n$) satisfies the recurrence relation
\begin{equation*}
   a_i = f_i \;\text{ for } 1\le i< k, \;\text{ and }\; a_n = \sum_{i=1}^{k-1} \binom{n}{i} a_{n-i} \text{ for } n\ge k,
\end{equation*}
where $f_i$ denotes the $i$-th Fubini number, cf.\ \cite[A276921]{oeis}. To verify this formula, observe that the set of active leaves can be organized by the size of the last block in the partition, so it is the union of $k-1$ disjoint sets. There are $\binom{n}{i}$ possible blocks of size $i$, and once such a block has been chosen as the last block of the partition, there are $a_{n-i}$ possible active partitions for the remaining elements.

We now derive a formula for the number $b_n$ of inactive chains at step $n$. The first stopping condition can only occur after $k$ steps, so $b_1=\cdots=b_{k-1}= 0$.

Inactive nodes of length $j\ge k$ correspond to ordered partitions of $[j]$ having a part of size $k$ that contains $j$, and such that all other parts have size less than $k$. There are $\binom{j-1}{k-1}$ ways to form the block of size $k$ that contains $j$, and for the remaining $j-k$ elements, we choose $i$ to be in parts to the left of that block. Thus,
\begin{equation*}
   \Delta_j = b_j-b_{j-1} = \binom{j-1}{k-1} \sum_{i=0}^{j-k} \binom{j-k}{i} a_i a_{j-k-i},
\end{equation*}
where $(a_n)_{n\in\mathbb{N}}$ is the sequence enumerating the active chains. Therefore, for $n\ge k$,
\begin{equation*}
  b_n = \sum_{j=k}^n \Delta_j = \sum_{j=k}^n \binom{j-1}{k-1} \sum_{i=0}^{j-k} \binom{j-k}{i} a_i a_{j-k-i}.
\end{equation*}

\medskip
Finally, the total number of weak-ordering chains subject to the above stopping condition can be obtained by adding the active and inactive chains. 
\begin{theorem}
If $w_n$ is the number of weak-ordering chains in $\W(n)$, subject to the stopping condition $x_{i_1}=\cdots=x_{i_k}$ with $k>1$, then $w_i = f_i$ for $1\le i < k$, where $f_i$ is the $i$\emph{th} Fubini number, and for $n\ge k$,
\[ w_n = \sum_{i=1}^{k-1} \binom{n}{i} a_{n-i} + \sum_{j=k}^n \binom{j-1}{k-1} 
 \sum_{i=0}^{j-k} \binom{j-k}{i} a_i a_{j-k-i}, \]
where $(a_n)_{n\in\mathbb{N}}$ is the sequence that counts the corresponding active chains.
\end{theorem}

For example, for $k=3$, we have $w_1=1$, $w_2=3$, and
\begin{equation*}
  w_n = n a_{n-1}+\binom{n}{2} a_{n-2} + \sum_{j=3}^n \binom{j-1}{2} \sum_{i=0}^{j-3} \binom{j-3}{i} a_i a_{j-3-i}
  \;\text{ for } n\ge 3. 
\end{equation*}
This sequence starts with $1, 3, 13, 73, 505, 4165, 39985, 438145,\dots$.

\subsection*{Stopping condition $\boldsymbol{x_{i_1}\!<x_{i_2}\!=x_{i_3}}$ with $\boldsymbol{i_1\!<i_2\!<i_3}$}
We start by discussing the number $a_n$ of active chains. Clearly, $a_1=1$ and $a_2=3$.

For $n\ge 3$, let $w\in\W(n)$ and let $\pi_w = \lbrace B_1, B_2, \dotsc, B_\ell \rbrace$ be its corresponding ordered partition of $[n]$. The chain $w$ is active if for every block $B_j$ with more than one element, its second largest element is smaller than all the elements in every block $B_i$ with $i<j$.

We enumerate these partitions by the size of their last block. If the size of $B_\ell$ is $k\ge 1$, then we must have $B_\ell = \{1, 2 , \dotsc, k-1, i_k\}$, with $k\le i_k \le n$. All other blocks must correspond to an active partition of the set $\{k,\dotsc, n\} \backslash \{i_k\}$.
So, we have the following recurrence relation for $n\ge 1$ (setting $a_0=1$):
\begin{equation*}
  a_n = \sum_{k=1}^n(n-k+1)a_{n-k} = \sum_{k=0}^{n-1}(k+1)a_{k}. 
\end{equation*}
Thus, $a_n = (n+1) a_{n-1}$ which implies $a_n=\frac{(n+1)!}{2}$. 

We now proceed to  derive a formula for $\Delta_n$, the number of inactive leaves at step $n$. The base cases are
$\Delta_1=\Delta_2=0$. Furthermore, we claim that 
\begin{equation}\label{eq:1<2=3rec}
  \Delta_n = n\Delta_{n-1}+(n-2)a_{n-2}, \; \text{ for } n\ge 3.
\end{equation}
First observe that, given an active partition $\pi_a$ of $[n-2]$, the partition obtained by adding the block $\lbrace n-1, n\rbrace$ to the end of $\pi_a$ is inactive. Moreover, $\pi_a$ can be used to create $n-3$ more inactive partitions as follows: Choose $i$ such that $2\le i \le n-2$. Then, for every element $j\ge i$ of $\pi_a$, replace $j$ by $j+1$ and place the block $\lbrace i, n\rbrace$ at the end of the modified partition. This gives $(n-2)a_{n-2}$ inactive partitions of $[n]$ having a last block of size 2.

The remaining inactive partitions must have a last block of size 1 or larger than 2. They can be generated from the $\Delta_{n-1}$ inactive partitions at level $n-1$ by the following process:

\begin{itemize}
\item For every inactive partition $\pi$ of $[n-1]$ and $i\in \lbrace 1, \dotsc, n-1\rbrace$, replace $j$ by $j+1$ for every $j\ge i$ and add the singleton $\{i\}$ to the end of the modified partition. This gives $(n-1)\Delta_{n-1}$ inactive partitions of $[n]$.
\item Alternatively, replacing $j$ by $j+1$ in $\pi$ and adding 1 to the last block, yields $\Delta_{n-1}$ additional inactive partitions of $[n]$.
\end{itemize}

In conclusion, we arrive at the recurrence relation \eqref{eq:1<2=3rec}. Since $a_n=\frac{(n+1)!}{2}$, we can solve \eqref{eq:1<2=3rec} and obtain the explicit formula
\begin{equation*}
  \Delta_n=\frac{n!}{2}\sum_{k=3}^n\frac{k-2}{k}, \; \text{ for } n\ge 3.
\end{equation*}

\subsection*{Stopping condition $\boldsymbol{x_{i_1}\!\le x_{i_2}\!=x_{i_3}}$ with $\boldsymbol{i_1\!<i_2\!<i_3}$}
As in the above case, we start by discussing the active chains. Once again, $a_1=1$ and $a_2=3$.

If $\pi = \lbrace B_1, B_2, \dotsc, B_\ell\rbrace$ is the partition corresponding to an active chain, we must have $\abs{B_\ell} \le 2$. Thus, the block $B_\ell$ is either a singleton $\lbrace i_k\rbrace$ or a block of the form $\lbrace 1, i_k\rbrace$. All other blocks must correspond to an active partition of either $\{1,\dotsc, n\} \backslash \{i_k\}$ or $\{2,\dotsc, n\} \backslash \{i_k\}$. We therefore have the recurrence relation (with $a_0=a_1=1$)
\begin{equation*}
  a_n = na_{n-1}+(n-1)a_{n-2} \;\text{ for } n\ge 2. 
\end{equation*}

For the enumeration of the chains that become inactive at step $n$, the base cases are $\Delta_1=\Delta_2=0$, and $\Delta_3=2$ (namely $x_1<x_2=x_3$ and $x_1=x_2=x_3$). For $n\ge 4$ we have
\begin{equation}\label{eq:1<=2=3rec}
  \Delta_n = (n-1)\Delta_{n-1}+(n-2)\Delta_{n-2} +a_{n-1}-a_{n-2}.
\end{equation}

To show this, we look at four disjoint cases:
\begin{enumerate}[\quad (a)]
\item Every inactive partition $\pi$ of $[n-1]$ yields $n-1$ inactive partitions of $[n]$ whose last block is a singleton:
Given $i\in \lbrace 1, \dotsc, n-1\rbrace$, replace $j$ by $j+1$ for every $j\ge i$ and add $\{i\}$ to the end. This gives $(n-1)\Delta_{n-1}$ inactive partitions of $[n]$.
\item For every inactive partition $\pi$ of $[n-2]$ relabel $j\to j+1$ and add the block $\{1\}$ to the end. Next, choose $i$ in $\{2,\dots,n-1\}$, replace $j$ by $j+1$ for every $j \ge i$, and add $i$ to the last block so it becomes $\{1, i\}$. There are $(n-2)\Delta_{n-2}$ such partitions of $[n]$.
\item For every active partition $\pi_a$ of $[n-2]$, the partition obtained by adding the block $\{n-1,n\}$ to the end of $\pi_a$ is inactive. Moreover, we can create $n-3$ additional inactive partitions of $[n]$ as follows: Choose $i$ such that $2\le i \le n-2$. For every $j\ge i$, replace $j$ by $j+1$ in $\pi_a$ and add the block $\{i, n\}$ to the end of the modified partition. Together, we have $(n-2)a_{n-2}$ inactive partitions whose last block is of size 2 and does not contain the element 1. 
\item Finally, every active partition $\pi_a$ of $[n-3]$ generates $n-2$ inactive partitions of $[n]$ obtained by relabeling $\pi_a$ so that we can add the block $\{1, i, n\}$ with $i\in\{2,\dots,n-1\}$ to the end of it. There are $(n-2)a_{n-3}$ inactive partitions of this type.
\end{enumerate}

The relation \eqref{eq:1<=2=3rec} follows from the fact that $(n-2)(a_{n-2}+a_{n-3}) = a_{n-1}-a_{n-2}$.

\end{document}